\documentclass[11pt]{article}
\usepackage{amssymb,amsfonts,amsmath,latexsym,epsf,tikz,url}
\usepackage[usenames,dvipsnames]{pstricks}
\usepackage{pstricks-add}
\usepackage{epsfig}
\usepackage{pst-grad} 
\usepackage{pst-plot} 
\usepackage[space]{grffile} 
\usepackage{etoolbox} 
\makeatletter 
\patchcmd\Gread@eps{\@inputcheck#1 }{\@inputcheck"#1"\relax}{}{}
\makeatother

\newtheorem{theorem}{Theorem}[section]

\newtheorem{corollary}[theorem]{Corollary}

\newtheorem{remark}[theorem]{Remark}

\newcommand{\qed}{\hfill $\square$\medskip}

\begin{document}

\title{Sombor index of certain graphs}

\author{
Nima Ghanbari \and Saeid Alikhani$^{}$\footnote{Corresponding author}
}

\date{\today}

\maketitle

\begin{center}
	Department of Informatics, University of Bergen, P.O. Box 7803, 5020 Bergen, Norway\\
Department of Mathematics, Yazd University, 89195-741, Yazd, Iran\\
{\tt  Nima.ghanbari@uib.no, alikhani@yazd.ac.ir }
\end{center}


\begin{abstract}
 Let $G=(V,E)$ be a finite simple graph. The Sombor  index $SO(G)$ of $G$
 is defined as $\sum_{uv\in E(G)}\sqrt{d_u^2+d_v^2}$, where $d_u$ is the degree of
 vertex $u$ in $G$. In this paper, we study this index for certain graphs and we  examine the effects on $SO(G)$ when $G$ is modified by operations on vertex and  edge of $G$. Also we  present bounds for the Sombor index of join and corona product  of two graphs.  
\end{abstract}


\medskip
\noindent{\bf AMS Subj.\ Class.:} 05C12

\section{Introduction}

Let $G = (V, E)$ be a finite, connected, simple graph. We denote the degree of a vertex $v$ in $G$ by $d_v$. 
A topological index of $G$ is a real number related to $G$. It does not depend on the labeling or pictorial representation of a graph. The Wiener index $W(G)$ is the first distance based topological index defined as $W(G) = \sum_{\{u,v\}\subseteq G}d(u,v)=\frac{1}{2} \sum_{u,v\in V(G)} d(u,v)$ with the summation runs over all pairs of vertices of $G$ \cite{20}.
The topological indices and graph invariants based on distances between vertices of a graph are widely used for characterizing molecular graphs, establishing relationships between structure and properties of molecules, predicting biological activity of chemical compounds, and making their chemical applications.  The
Wiener index is one of the most used topological indices with high correlation with many physical and chemical indices of molecular compounds \cite{20}. 
Recently in \cite{Gutman2}  a new vertex-degree-based molecular structure descriptor was put forward, the Sombor index, defined as  
$$SO(G) =\sum_{uv\in E(G)}\sqrt{d_u^2+d_v^2}.$$
Cruz, Gutman and Rada in \cite{AMC} characterized the graphs extremal  with respect to this index over the chemical graphs, chemical trees and hexagon systems. 

The  corona of two graphs $G_1$ and $G_2$, is the graph
$G_1 \circ G_2$ formed from one copy of $G_1$ and $|V(G_1)|$ copies of $G_2$,
where the ith vertex of $G_1$ is adjacent to every vertex in the ith copy of $G_2$.
The corona $G\circ K_1$, in particular, is the graph constructed from a copy of $G$,
where for each vertex $v\in V(G)$, a new vertex $v'$ and a pendant edge $vv'$ are added.
The  join of two graphs $G_1$ and $G_2$, denoted by $G_1\vee G_2$,
is a graph with vertex set  $V(G_1)\cup V(G_2)$
and edge set $E(G_1)\cup E(G_2)\cup \{uv| u\in V(G_1)$ and $v\in V(G_2)\}$.
 The Cartesian product $G\Box H$ of graphs $G$ and $H$ is a graph such that
 the vertex set of $G\Box H$ is the Cartesian product $V(G)\times V(H)$,  and
 two vertices $(u,u')$ and $(v,v')$ are adjacent in $G\Box H$ if and only if either
 $u=v$ and $u'$ is adjacent to $v'$ in $H$, or
 $u'=v'$ and $u$ is adjacent to $v$ in $G$.

\medskip

In the next section, we compute the Sombor index for special graphs, cactus chains and grid graphs. In Section 3, we  we  examine the effects on $SO(G)$ when $G$ is modified by operations on vertex and  edge of $G$ and finally in Section 4, we study the Sombor index join and corona of two graphs.

\section{Sombor index of certain graphs}
In this section, we compute the Sombor index for certain graphs, such as paths, cycles, friendship graph, grid graphs and cactus chains. 
	
\subsection{Sombor index of specific graphs}
 We begin with the following theorem:  
	\begin{theorem}\label{specific} 
		\begin{enumerate} 
\item[(i)] 
For every $n\in \mathbb{N}-\{1,2\}$, $SO(P_n)=2\sqrt{5}+(2n-6)\sqrt{2},$  for  $n\geq 2$,
$SO(K_n)=\frac{n(n-1)^2}{2}\sqrt{2}.$
\item[(ii)] 
For the cycle graph $C_n$, $SO(C_n)=2n\sqrt{2}.$
\item[(iii)] 
For  every $m,n\in \mathbb{N}$, $SO(K_{1,n})=n\sqrt{n^2+1},$ and $SO(K_{m,n})=mn\sqrt{m^2+n^2}.$
\item[(iv)] 
For the wheel graph $W_n=C_{n-1}\vee K_1$, $SO(W_n)=(3n-3)\sqrt{2}+(n-1)\sqrt{9+(n-1)^2}.$
\item[(v)] 
For the ladder graph $L_n=P_n\Box K_2$, ($n\geq 3$), $SO(L_n)=(9n-22)\sqrt{2}+4\sqrt{13}.$
	\item[(vi)] 
	For the friendship graph $F_n=K_1\vee nK_2$, $SO(F_n)=2n\sqrt{2}+4n\sqrt{n^2+1}.$

	\item[(vii)]
		For the book graph $B_n=K_{1,n}\square K_2$, and for $n\geq 3$,
		$$SO(B_n)=(3n-1)\sqrt{2}+2n\sqrt{4+(n-1)^2}.$$
	
		\item[(viii)] For the Dutch windmill graph $D_n^{(m)}$, and for every $n\geq 3$ and $m\geq 2$,
		$$SO(D_n^{(m)})=2m(n-2)\sqrt{2}+4m\sqrt{m^2+1}.$$
\end{enumerate} 
	\end{theorem}
		
	\begin{proof}
	The proof of all parts are easy and similar. For instance we state the proof of Part (i). 
	\begin{enumerate} 
		\item[(i)]
			There are two edges with endpoints of degree $1$ and $2$. Also there are $n-3$ edges with endpoints of degree $2$. Therefore 
	$$SO(P_n)=2\sqrt{1+4}+(n-3)\sqrt{4+4},$$
	and we have the result. \qed
\end{enumerate}
	\end{proof}
	We have the following corollary: 
\begin{corollary}
	The natural numbers $1,2,...,59$ cannot be the value of Sombor index of any graph. In other words, the smallest natural number as the Sombor index of a graph is $60$. 	\end{corollary}
\begin{proof}
	The minimum value of $\sqrt{a^2+b^2}$ as a natural number (when $a$ and $b$ are natural numbers) is $5$ and it is happen when $a=4$ and $b=3$. So we need at least $5$ vertices to have such a graph. It is  easy to see that  the graphs of order $5$ and $6$ do not have natural Sombor index. On the other hand, all the values of  $\sqrt{a^2+b^2}$ in the definition of Sombor index should be natural numbers to have a natural Sombor index. The complete bipartite graph $K_{3,4}$ has this condition and by the Theorem \ref{specific}(iii) we have the result. By an easy check, we observe that no other graph with $7$ vertices has natural Sombor index and also there is no graph with natural Sombor index less than $60$. 
	\qed
\end{proof}

\begin{figure}
	\begin{center}
		\psscalebox{0.5 0.5}
		{
			\begin{pspicture}(0,-7.6)(12.86,3.66)
			\psdots[linecolor=black, dotsize=0.4](1.63,2.03)
			\psdots[linecolor=black, dotsize=0.4](3.23,2.03)
			\psdots[linecolor=black, dotsize=0.4](4.83,2.03)
			\psdots[linecolor=black, dotsize=0.4](6.43,2.03)
			\psdots[linecolor=black, dotsize=0.4](8.03,2.03)
			\psline[linecolor=black, linewidth=0.08](1.63,2.03)(8.43,2.03)(8.43,2.03)
			\psdots[linecolor=black, dotsize=0.1](8.83,2.03)
			\psdots[linecolor=black, dotsize=0.1](9.23,2.03)
			\psdots[linecolor=black, dotsize=0.1](9.63,2.03)
			\psline[linecolor=black, linewidth=0.08](10.03,2.03)(11.63,2.03)(11.63,2.03)
			\psdots[linecolor=black, dotsize=0.4](10.43,2.03)
			\psdots[linecolor=black, dotsize=0.4](12.03,2.03)
			\psline[linecolor=black, linewidth=0.08](10.43,2.03)(12.03,2.03)(12.03,2.03)
			\psdots[linecolor=black, dotsize=0.4](1.63,0.43)
			\psdots[linecolor=black, dotsize=0.4](3.23,0.43)
			\psdots[linecolor=black, dotsize=0.4](4.83,0.43)
			\psdots[linecolor=black, dotsize=0.4](6.43,0.43)
			\psdots[linecolor=black, dotsize=0.4](8.03,0.43)
			\psdots[linecolor=black, dotsize=0.4](10.43,0.43)
			\psdots[linecolor=black, dotsize=0.4](12.03,0.43)
			\psdots[linecolor=black, dotsize=0.4](1.63,-1.17)
			\psdots[linecolor=black, dotsize=0.4](3.23,-1.17)
			\psdots[linecolor=black, dotsize=0.4](4.83,-1.17)
			\psdots[linecolor=black, dotsize=0.4](6.43,-1.17)
			\psdots[linecolor=black, dotsize=0.4](8.03,-1.17)
			\psdots[linecolor=black, dotsize=0.4](10.43,-1.17)
			\psdots[linecolor=black, dotsize=0.4](12.03,-1.17)
			\psdots[linecolor=black, dotsize=0.4](1.63,-2.77)
			\psdots[linecolor=black, dotsize=0.4](3.23,-2.77)
			\psdots[linecolor=black, dotsize=0.4](4.83,-2.77)
			\psdots[linecolor=black, dotsize=0.4](6.43,-2.77)
			\psdots[linecolor=black, dotsize=0.4](8.03,-2.77)
			\psdots[linecolor=black, dotsize=0.4](10.43,-2.77)
			\psdots[linecolor=black, dotsize=0.4](12.03,-2.77)
			\psline[linecolor=black, linewidth=0.08](1.63,-3.17)(1.63,2.03)(1.63,2.03)
			\psline[linecolor=black, linewidth=0.08](3.23,-3.17)(3.23,2.03)(3.23,2.03)
			\psline[linecolor=black, linewidth=0.08](4.83,-3.17)(4.83,2.03)(4.83,2.03)
			\psline[linecolor=black, linewidth=0.08](6.43,-3.17)(6.43,2.03)(6.43,2.03)
			\psline[linecolor=black, linewidth=0.08](8.03,-3.17)(8.03,2.03)(8.03,2.03)
			\psline[linecolor=black, linewidth=0.08](10.43,-3.17)(10.43,2.03)(10.43,2.03)
			\psline[linecolor=black, linewidth=0.08](12.03,-3.17)(12.03,2.03)(12.03,2.03)
			\psline[linecolor=black, linewidth=0.08](8.43,0.43)(1.63,0.43)(1.63,0.43)
			\psline[linecolor=black, linewidth=0.08](10.03,0.43)(12.03,0.43)(12.03,0.43)
			\psline[linecolor=black, linewidth=0.08](10.03,-1.17)(12.03,-1.17)(12.03,-1.17)
			\psline[linecolor=black, linewidth=0.08](10.03,-2.77)(12.03,-2.77)(12.03,-2.77)
			\psline[linecolor=black, linewidth=0.08](8.43,-1.17)(1.63,-1.17)(1.63,-1.17)
			\psline[linecolor=black, linewidth=0.08](8.43,-2.77)(1.63,-2.77)(1.63,-2.77)
			\psline[linecolor=black, linewidth=0.08](12.03,-4.77)(12.03,-6.77)(10.03,-6.77)(10.03,-6.77)
			\psline[linecolor=black, linewidth=0.08](10.43,-4.77)(10.43,-6.77)(10.43,-6.77)
			\psline[linecolor=black, linewidth=0.08](10.03,-5.17)(12.03,-5.17)(12.03,-5.17)
			\psline[linecolor=black, linewidth=0.08](8.43,-5.17)(1.63,-5.17)(1.63,-5.17)
			\psline[linecolor=black, linewidth=0.08](1.63,-4.77)(1.63,-6.77)(8.43,-6.77)(8.43,-6.77)
			\psline[linecolor=black, linewidth=0.08](8.03,-4.77)(8.03,-6.77)(8.03,-6.77)
			\psline[linecolor=black, linewidth=0.08](6.43,-4.77)(6.43,-6.77)(6.43,-6.77)
			\psline[linecolor=black, linewidth=0.08](4.83,-4.77)(4.83,-6.77)(4.83,-6.77)
			\psline[linecolor=black, linewidth=0.08](3.23,-4.77)(3.23,-6.77)(3.23,-6.77)
			\psdots[linecolor=black, dotsize=0.4](12.03,-5.17)
			\psdots[linecolor=black, dotsize=0.4](10.43,-5.17)
			\psdots[linecolor=black, dotsize=0.4](10.43,-6.77)
			\psdots[linecolor=black, dotsize=0.4](12.03,-6.77)
			\psdots[linecolor=black, dotsize=0.4](8.03,-5.17)
			\psdots[linecolor=black, dotsize=0.4](6.43,-5.17)
			\psdots[linecolor=black, dotsize=0.4](4.83,-5.17)
			\psdots[linecolor=black, dotsize=0.4](3.23,-5.17)
			\psdots[linecolor=black, dotsize=0.4](1.63,-5.17)
			\psdots[linecolor=black, dotsize=0.4](1.63,-6.77)
			\psdots[linecolor=black, dotsize=0.4](3.23,-6.77)
			\psdots[linecolor=black, dotsize=0.4](4.83,-6.77)
			\psdots[linecolor=black, dotsize=0.4](6.43,-6.77)
			\psdots[linecolor=black, dotsize=0.4](8.03,-6.77)
			\psdots[linecolor=black, dotsize=0.1](8.83,0.43)
			\psdots[linecolor=black, dotsize=0.1](9.23,0.43)
			\psdots[linecolor=black, dotsize=0.1](9.63,0.43)
			\psdots[linecolor=black, dotsize=0.1](8.83,-1.17)
			\psdots[linecolor=black, dotsize=0.1](9.23,-1.17)
			\psdots[linecolor=black, dotsize=0.1](9.63,-1.17)
			\psdots[linecolor=black, dotsize=0.1](8.83,-2.77)
			\psdots[linecolor=black, dotsize=0.1](9.23,-2.77)
			\psdots[linecolor=black, dotsize=0.1](9.63,-2.77)
			\psdots[linecolor=black, dotsize=0.1](10.43,-3.57)
			\psdots[linecolor=black, dotsize=0.1](10.43,-3.97)
			\psdots[linecolor=black, dotsize=0.1](10.43,-4.37)
			\psdots[linecolor=black, dotsize=0.1](12.03,-3.57)
			\psdots[linecolor=black, dotsize=0.1](12.03,-3.97)
			\psdots[linecolor=black, dotsize=0.1](12.03,-4.37)
			\psdots[linecolor=black, dotsize=0.1](8.83,-5.17)
			\psdots[linecolor=black, dotsize=0.1](9.23,-5.17)
			\psdots[linecolor=black, dotsize=0.1](9.63,-5.17)
			\psdots[linecolor=black, dotsize=0.1](8.83,-6.77)
			\psdots[linecolor=black, dotsize=0.1](9.23,-6.77)
			\psdots[linecolor=black, dotsize=0.1](9.63,-6.77)
			\psdots[linecolor=black, dotsize=0.1](8.03,-3.57)
			\psdots[linecolor=black, dotsize=0.1](8.03,-3.97)
			\psdots[linecolor=black, dotsize=0.1](8.03,-4.37)
			\psdots[linecolor=black, dotsize=0.1](6.43,-3.57)
			\psdots[linecolor=black, dotsize=0.1](6.43,-3.97)
			\psdots[linecolor=black, dotsize=0.1](6.43,-4.37)
			\psdots[linecolor=black, dotsize=0.1](4.83,-3.57)
			\psdots[linecolor=black, dotsize=0.1](4.83,-3.97)
			\psdots[linecolor=black, dotsize=0.1](4.83,-4.37)
			\psdots[linecolor=black, dotsize=0.1](3.23,-3.57)
			\psdots[linecolor=black, dotsize=0.1](3.23,-3.97)
			\psdots[linecolor=black, dotsize=0.1](3.23,-4.37)
			\psdots[linecolor=black, dotsize=0.1](1.63,-3.57)
			\psdots[linecolor=black, dotsize=0.1](1.63,-3.97)
			\psdots[linecolor=black, dotsize=0.1](1.63,-4.37)
			\psdots[linecolor=black, dotsize=0.1](8.83,-3.57)
			\psdots[linecolor=black, dotsize=0.1](9.23,-3.97)
			\psdots[linecolor=black, dotsize=0.1](9.63,-4.37)
			\rput[bl](7.096667,3.0166667){$P_n$}
			\rput[bl](0.23,-2.33){$P_m$}
			\psline[linecolor=blue, linewidth=0.04, linestyle=dotted, dotsep=0.10583334cm](0.83,3.63)(2.43,3.63)(2.43,-7.57)(0.83,-7.57)(0.83,3.63)(0.83,3.63)
			\psline[linecolor=red, linewidth=0.04, linestyle=dotted, dotsep=0.10583334cm](0.03,2.83)(12.83,2.83)(12.83,1.23)(0.03,1.23)(0.03,2.83)(0.03,2.83)
			\end{pspicture}
		}
	\end{center}
	\caption{Grid graph $P_m \Box P_n$ } \label{Gird}
\end{figure}
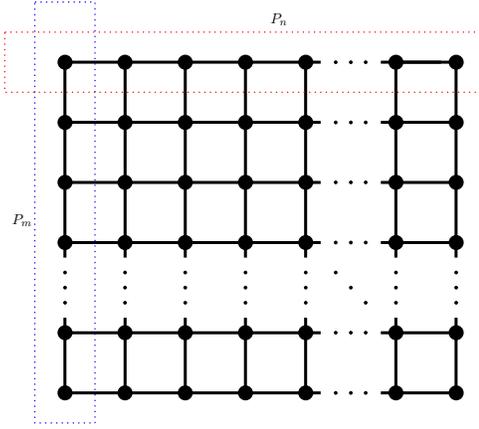

Here, we consider the grid graph ($P_n\Box P_m$), and obtain  its Sombor index.  
\begin{theorem}
	Let $P_m\Box  P_n$ be the gird graph (Figure \ref{Gird}). For every $n\geq 6$ and $m \geq 6$,
	$$SO(P_m\Box P_n)=(8nm-17n-17m-60)\sqrt{2}+4\sqrt{13}+10(m+n-4).$$
\end{theorem}

\begin{proof}
	There are four edges with endpoints of degree  $2$ and $3$ and there are $m+n-4$ edges with endpoints of degree $3$. Also there are $2m+2n-8$ edges with endpoints of degree $3$ and $4$ and there are $2nm-5n-5m-12$ edges with endpoints of degree $4$. Therefore 
	\begin{align*}
	SO(P_m\Box P_n)&=4\sqrt{4+9}+(m+n-4)\sqrt{9+9}(2m+2n-8)\sqrt{9+16}\\
	&\quad+(2nm-5n-5m-12)\sqrt{16+16},
	\end{align*}
	and we have the result.			
	\qed
\end{proof}

\subsection{Sombor index of cactus chains} 	
	
In this subsection,  we consider a  class of simple linear polymers called cactus chains. Cactus graphs were first known as Husimi tree, they appeared in the scientific literature some sixty years ago in papers by Husimi and
Riddell concerned with cluster integrals in the theory of condensation in statistical mechanics \cite{9,12,14}. 
We refer the reader to papers \cite{chellali,13} for some aspects of parameters of  cactus graphs.

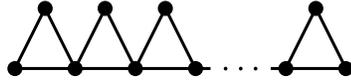
\begin{figure}
	\begin{center}
		\psscalebox{0.5 0.5}
		{
			\begin{pspicture}(0,-7.2)(9.194231,-5.205769)
			\psdots[linecolor=black, dotsize=0.1](5.7971153,-7.0028844)
			\psdots[linecolor=black, dotsize=0.1](6.1971154,-7.0028844)
			\psdots[linecolor=black, dotsize=0.1](6.5971155,-7.0028844)
			\psdots[linecolor=black, dotsize=0.4](7.3971157,-7.0028844)
			\psdots[linecolor=black, dotsize=0.4](8.197116,-5.4028845)
			\psdots[linecolor=black, dotsize=0.4](8.997115,-7.0028844)
			\psdots[linecolor=black, dotsize=0.4](4.9971156,-7.0028844)
			\psdots[linecolor=black, dotsize=0.4](4.1971154,-5.4028845)
			\psdots[linecolor=black, dotsize=0.4](3.3971155,-7.0028844)
			\psdots[linecolor=black, dotsize=0.4](2.5971155,-5.4028845)
			\psdots[linecolor=black, dotsize=0.4](1.7971154,-7.0028844)
			\psdots[linecolor=black, dotsize=0.4](0.9971155,-5.4028845)
			\psdots[linecolor=black, dotsize=0.4](0.19711548,-7.0028844)
			\psline[linecolor=black, linewidth=0.08](0.19711548,-7.0028844)(4.9971156,-7.0028844)(4.1971154,-5.4028845)(3.3971155,-7.0028844)(2.5971155,-5.4028845)(1.7971154,-7.0028844)(0.9971155,-5.4028845)(0.19711548,-7.0028844)(0.19711548,-7.0028844)
			\psline[linecolor=black, linewidth=0.08](7.3971157,-7.0028844)(8.197116,-5.4028845)(8.997115,-7.0028844)(7.3971157,-7.0028844)(7.3971157,-7.0028844)
			\psline[linecolor=black, linewidth=0.08](4.9971156,-7.0028844)(5.3971157,-7.0028844)(5.3971157,-7.0028844)
			\psline[linecolor=black, linewidth=0.08](7.3971157,-7.0028844)(6.9971156,-7.0028844)(6.9971156,-7.0028844)
			\end{pspicture}
		}
	\end{center}
	\caption{Chain triangular cactus $T_n$} \label{Chaintri}
\end{figure}

\begin{figure}
	\begin{center}
		\psscalebox{0.5 0.5}
		{
			\begin{pspicture}(0,-8.0)(9.194231,-4.405769)
			\psdots[linecolor=black, dotsize=0.1](5.7971153,-6.2028847)
			\psdots[linecolor=black, dotsize=0.1](6.1971154,-6.2028847)
			\psdots[linecolor=black, dotsize=0.1](6.5971155,-6.2028847)
			\psdots[linecolor=black, dotsize=0.4](7.3971157,-6.2028847)
			\psdots[linecolor=black, dotsize=0.4](8.197116,-4.6028843)
			\psdots[linecolor=black, dotsize=0.4](8.997115,-6.2028847)
			\psdots[linecolor=black, dotsize=0.4](4.9971156,-6.2028847)
			\psdots[linecolor=black, dotsize=0.4](4.1971154,-4.6028843)
			\psdots[linecolor=black, dotsize=0.4](3.3971155,-6.2028847)
			\psdots[linecolor=black, dotsize=0.4](2.5971155,-4.6028843)
			\psdots[linecolor=black, dotsize=0.4](1.7971154,-6.2028847)
			\psdots[linecolor=black, dotsize=0.4](0.9971155,-4.6028843)
			\psdots[linecolor=black, dotsize=0.4](0.19711548,-6.2028847)
			\psdots[linecolor=black, dotsize=0.4](0.9971155,-7.8028846)
			\psdots[linecolor=black, dotsize=0.4](2.5971155,-7.8028846)
			\psdots[linecolor=black, dotsize=0.4](4.1971154,-7.8028846)
			\psdots[linecolor=black, dotsize=0.4](8.197116,-7.8028846)
			\psline[linecolor=black, linewidth=0.08](7.3971157,-6.2028847)(8.197116,-4.6028843)(8.997115,-6.2028847)(8.197116,-7.8028846)(7.3971157,-6.2028847)(7.3971157,-6.2028847)
			\psline[linecolor=black, linewidth=0.08](4.9971156,-6.2028847)(4.1971154,-4.6028843)(3.3971155,-6.2028847)(2.5971155,-4.6028843)(1.7971154,-6.2028847)(0.9971155,-4.6028843)(0.19711548,-6.2028847)(0.9971155,-7.8028846)(1.7971154,-6.2028847)(2.5971155,-7.8028846)(3.3971155,-6.2028847)(4.1971154,-7.8028846)(4.9971156,-6.2028847)(4.9971156,-6.2028847)
			\end{pspicture}
		}
	\end{center}
	\caption{Para-chain square cactus $Q_n$} \label{paraChainsqu}
\end{figure}
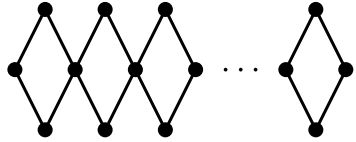

\begin{figure}
\begin{minipage}{7.5cm}
		\psscalebox{0.45 0.45}
		{
			\begin{pspicture}(0,-6.0)(12.394231,-2.405769)
			\psdots[linecolor=black, dotsize=0.4](0.19711548,-2.6028845)
			\psdots[linecolor=black, dotsize=0.4](1.7971154,-2.6028845)
			\psdots[linecolor=black, dotsize=0.4](0.19711548,-4.2028847)
			\psdots[linecolor=black, dotsize=0.4](1.7971154,-4.2028847)
			\psdots[linecolor=black, dotsize=0.4](3.3971155,-4.2028847)
			\psdots[linecolor=black, dotsize=0.4](4.9971156,-4.2028847)
			\psdots[linecolor=black, dotsize=0.4](6.5971155,-4.2028847)
			\psdots[linecolor=black, dotsize=0.4](1.7971154,-5.8028846)
			\psdots[linecolor=black, dotsize=0.4](3.3971155,-5.8028846)
			\psdots[linecolor=black, dotsize=0.4](3.3971155,-2.6028845)
			\psdots[linecolor=black, dotsize=0.4](4.9971156,-2.6028845)
			\psdots[linecolor=black, dotsize=0.4](6.5971155,-5.8028846)
			\psdots[linecolor=black, dotsize=0.4](4.9971156,-5.8028846)
			\psline[linecolor=black, linewidth=0.08](0.19711548,-4.2028847)(6.5971155,-4.2028847)(6.5971155,-4.2028847)
			\psline[linecolor=black, linewidth=0.08](0.19711548,-2.6028845)(1.7971154,-2.6028845)(1.7971154,-5.8028846)(3.3971155,-5.8028846)(3.3971155,-2.6028845)(4.9971156,-2.6028845)(4.9971156,-5.8028846)(6.5971155,-5.8028846)(6.5971155,-4.2028847)(6.5971155,-4.2028847)
			\psline[linecolor=black, linewidth=0.08](0.19711548,-4.2028847)(0.19711548,-2.6028845)(0.19711548,-2.6028845)
			\psline[linecolor=black, linewidth=0.08](6.9971156,-4.2028847)(6.5971155,-4.2028847)(6.5971155,-4.2028847)
			\psline[linecolor=black, linewidth=0.08](8.5971155,-4.2028847)(8.997115,-4.2028847)(8.997115,-4.2028847)
			\psdots[linecolor=black, dotsize=0.4](8.997115,-4.2028847)
			\psdots[linecolor=black, dotsize=0.4](8.997115,-2.6028845)
			\psdots[linecolor=black, dotsize=0.4](10.5971155,-2.6028845)
			\psdots[linecolor=black, dotsize=0.4](10.5971155,-4.2028847)
			\psdots[linecolor=black, dotsize=0.4](10.5971155,-5.8028846)
			\psdots[linecolor=black, dotsize=0.4](12.197116,-5.8028846)
			\psdots[linecolor=black, dotsize=0.4](12.197116,-4.2028847)
			\psline[linecolor=black, linewidth=0.08](8.997115,-4.2028847)(12.197116,-4.2028847)(12.197116,-5.8028846)(10.5971155,-5.8028846)(10.5971155,-2.6028845)(8.997115,-2.6028845)(8.997115,-4.2028847)(8.997115,-4.2028847)
			\psdots[linecolor=black, dotsize=0.1](7.3971157,-4.2028847)
			\psdots[linecolor=black, dotsize=0.1](7.7971153,-4.2028847)
			\psdots[linecolor=black, dotsize=0.1](8.197116,-4.2028847)
			\end{pspicture}
		}
\end{minipage}
\begin{minipage}{7.5cm}
		\psscalebox{0.45 0.45}
		{
			\begin{pspicture}(0,-6.8)(13.194231,-1.605769)
			\psdots[linecolor=black, dotsize=0.4](2.1971154,-1.8028846)
			\psdots[linecolor=black, dotsize=0.4](2.1971154,-4.2028847)
			\psdots[linecolor=black, dotsize=0.4](2.5971155,-3.0028846)
			\psdots[linecolor=black, dotsize=0.4](3.3971155,-3.0028846)
			\psdots[linecolor=black, dotsize=0.4](3.7971156,-4.2028847)
			\psdots[linecolor=black, dotsize=0.4](3.7971156,-1.8028846)
			\psdots[linecolor=black, dotsize=0.4](5.3971157,-1.8028846)
			\psdots[linecolor=black, dotsize=0.4](5.7971153,-3.0028846)
			\psdots[linecolor=black, dotsize=0.4](5.3971157,-4.2028847)
			\psdots[linecolor=black, dotsize=0.4](0.59711546,-1.8028846)
			\psdots[linecolor=black, dotsize=0.4](0.19711548,-3.0028846)
			\psdots[linecolor=black, dotsize=0.4](0.59711546,-4.2028847)
			\psdots[linecolor=black, dotsize=0.4](1.7971154,-5.4028845)
			\psdots[linecolor=black, dotsize=0.4](2.1971154,-6.6028843)
			\psdots[linecolor=black, dotsize=0.4](3.7971156,-6.6028843)
			\psdots[linecolor=black, dotsize=0.4](4.1971154,-5.4028845)
			\psdots[linecolor=black, dotsize=0.4](6.9971156,-4.2028847)
			\psdots[linecolor=black, dotsize=0.4](4.9971156,-5.4028845)
			\psdots[linecolor=black, dotsize=0.4](5.3971157,-6.6028843)
			\psdots[linecolor=black, dotsize=0.4](7.3971157,-5.4028845)
			\psdots[linecolor=black, dotsize=0.4](6.9971156,-6.6028843)
			\psdots[linecolor=black, dotsize=0.4](10.997115,-1.8028846)
			\psdots[linecolor=black, dotsize=0.4](10.997115,-4.2028847)
			\psdots[linecolor=black, dotsize=0.4](11.397116,-3.0028846)
			\psdots[linecolor=black, dotsize=0.4](12.5971155,-4.2028847)
			\psdots[linecolor=black, dotsize=0.4](9.397116,-1.8028846)
			\psdots[linecolor=black, dotsize=0.4](8.997115,-3.0028846)
			\psdots[linecolor=black, dotsize=0.4](9.397116,-4.2028847)
			\psdots[linecolor=black, dotsize=0.4](10.5971155,-5.4028845)
			\psdots[linecolor=black, dotsize=0.4](10.997115,-6.6028843)
			\psdots[linecolor=black, dotsize=0.4](12.5971155,-6.6028843)
			\psdots[linecolor=black, dotsize=0.4](12.997115,-5.4028845)
			\psline[linecolor=black, linewidth=0.08](6.9971156,-4.2028847)(0.59711546,-4.2028847)(0.19711548,-3.0028846)(0.59711546,-1.8028846)(2.1971154,-1.8028846)(2.5971155,-3.0028846)(2.1971154,-4.2028847)(1.7971154,-5.4028845)(2.1971154,-6.6028843)(3.7971156,-6.6028843)(4.1971154,-5.4028845)(3.7971156,-4.2028847)(3.3971155,-3.0028846)(3.7971156,-1.8028846)(5.3971157,-1.8028846)(5.7971153,-3.0028846)(5.3971157,-4.2028847)(4.9971156,-5.4028845)(5.3971157,-6.6028843)(6.9971156,-6.6028843)(7.3971157,-5.4028845)(6.9971156,-4.2028847)(6.9971156,-4.2028847)
			\psline[linecolor=black, linewidth=0.08](9.397116,-4.2028847)(12.5971155,-4.2028847)(12.997115,-5.4028845)(12.5971155,-6.6028843)(10.997115,-6.6028843)(10.5971155,-5.4028845)(11.397116,-3.0028846)(10.997115,-1.8028846)(9.397116,-1.8028846)(8.997115,-3.0028846)(9.397116,-4.2028847)(9.397116,-4.2028847)
			\psline[linecolor=black, linewidth=0.08](7.3971157,-4.2028847)(6.9971156,-4.2028847)(6.9971156,-4.2028847)
			\psline[linecolor=black, linewidth=0.08](8.997115,-4.2028847)(9.397116,-4.2028847)(9.397116,-4.2028847)
			\psdots[linecolor=black, dotsize=0.1](8.197116,-4.2028847)
			\psdots[linecolor=black, dotsize=0.1](7.7971153,-4.2028847)
			\psdots[linecolor=black, dotsize=0.1](8.5971155,-4.2028847)
			\end{pspicture}
		}
	\end{minipage}
	\caption{Para-chain square cactus $O_n$ and Ortho-chain graph $O_n^h$ } \label{ortho-ohn}
\end{figure}
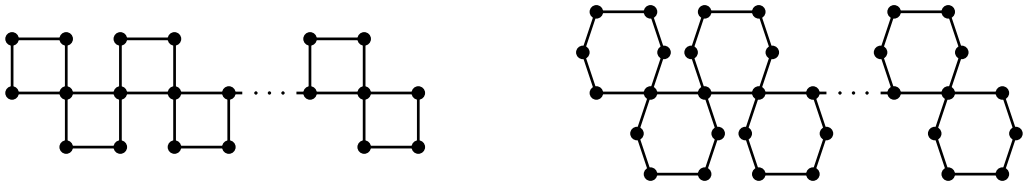

\begin{figure}
\begin{minipage}{7.5cm}
		\psscalebox{0.45 0.45}
		{
			\begin{pspicture}(0,-5.6)(16.794231,-2.805769)
			\psdots[linecolor=black, dotsize=0.4](2.1971154,-3.0028846)
			\psdots[linecolor=black, dotsize=0.4](2.1971154,-5.4028845)
			\psdots[linecolor=black, dotsize=0.4](2.9971154,-4.2028847)
			\psdots[linecolor=black, dotsize=0.4](2.9971154,-4.2028847)
			\psdots[linecolor=black, dotsize=0.4](3.7971153,-5.4028845)
			\psdots[linecolor=black, dotsize=0.4](3.7971153,-3.0028846)
			\psdots[linecolor=black, dotsize=0.4](4.9971156,-3.0028846)
			\psdots[linecolor=black, dotsize=0.4](5.7971153,-4.2028847)
			\psdots[linecolor=black, dotsize=0.4](4.9971156,-5.4028845)
			\psdots[linecolor=black, dotsize=0.4](0.9971154,-3.0028846)
			\psdots[linecolor=black, dotsize=0.4](0.19711538,-4.2028847)
			\psdots[linecolor=black, dotsize=0.4](0.9971154,-5.4028845)
			\psdots[linecolor=black, dotsize=0.4](7.7971153,-3.0028846)
			\psdots[linecolor=black, dotsize=0.4](7.7971153,-5.4028845)
			\psdots[linecolor=black, dotsize=0.4](8.5971155,-4.2028847)
			\psdots[linecolor=black, dotsize=0.4](8.5971155,-4.2028847)
			\psdots[linecolor=black, dotsize=0.4](6.5971155,-3.0028846)
			\psdots[linecolor=black, dotsize=0.4](5.7971153,-4.2028847)
			\psdots[linecolor=black, dotsize=0.4](6.5971155,-5.4028845)
			\psdots[linecolor=black, dotsize=0.4](12.997115,-3.0028846)
			\psdots[linecolor=black, dotsize=0.4](12.997115,-5.4028845)
			\psdots[linecolor=black, dotsize=0.4](13.797115,-4.2028847)
			\psdots[linecolor=black, dotsize=0.4](13.797115,-4.2028847)
			\psdots[linecolor=black, dotsize=0.4](14.5971155,-5.4028845)
			\psdots[linecolor=black, dotsize=0.4](14.5971155,-3.0028846)
			\psdots[linecolor=black, dotsize=0.4](15.797115,-3.0028846)
			\psdots[linecolor=black, dotsize=0.4](16.597115,-4.2028847)
			\psdots[linecolor=black, dotsize=0.4](15.797115,-5.4028845)
			\psdots[linecolor=black, dotsize=0.4](11.797115,-3.0028846)
			\psdots[linecolor=black, dotsize=0.4](10.997115,-4.2028847)
			\psdots[linecolor=black, dotsize=0.4](11.797115,-5.4028845)
			\psdots[linecolor=black, dotsize=0.4](8.5971155,-4.2028847)
			\psdots[linecolor=black, dotsize=0.4](8.5971155,-4.2028847)
			\psline[linecolor=black, linewidth=0.08](0.19711538,-4.2028847)(0.9971154,-3.0028846)(2.1971154,-3.0028846)(2.9971154,-4.2028847)(3.7971153,-3.0028846)(4.9971156,-3.0028846)(5.7971153,-4.2028847)(6.5971155,-3.0028846)(7.7971153,-3.0028846)(8.5971155,-4.2028847)(7.7971153,-5.4028845)(6.5971155,-5.4028845)(5.7971153,-4.2028847)(4.9971156,-5.4028845)(3.7971153,-5.4028845)(2.9971154,-4.2028847)(2.1971154,-5.4028845)(0.9971154,-5.4028845)(0.19711538,-4.2028847)(0.19711538,-4.2028847)
			\psline[linecolor=black, linewidth=0.08](10.997115,-4.2028847)(11.797115,-3.0028846)(12.997115,-3.0028846)(13.797115,-4.2028847)(14.5971155,-3.0028846)(15.797115,-3.0028846)(16.597115,-4.2028847)(15.797115,-5.4028845)(14.5971155,-5.4028845)(13.797115,-4.2028847)(12.997115,-5.4028845)(11.797115,-5.4028845)(10.997115,-4.2028847)(10.997115,-4.2028847)
			\psdots[linecolor=black, dotsize=0.1](9.397116,-4.2028847)
			\psdots[linecolor=black, dotsize=0.1](9.797115,-4.2028847)
			\psdots[linecolor=black, dotsize=0.1](10.197115,-4.2028847)
			\end{pspicture}
		}
	\end{minipage}
	\hspace{1.02cm}
	\begin{minipage}{7.5cm} 
		\psscalebox{0.45 0.40}
		{
			\begin{pspicture}(0,-6.4)(12.394231,-0.40576905)
			\psdots[linecolor=black, dotsize=0.4](0.9971155,-0.60288453)
			\psdots[linecolor=black, dotsize=0.4](1.7971154,-1.8028846)
			\psdots[linecolor=black, dotsize=0.4](1.7971154,-3.4028845)
			\psdots[linecolor=black, dotsize=0.4](0.9971155,-4.6028843)
			\psdots[linecolor=black, dotsize=0.4](0.19711548,-1.8028846)
			\psdots[linecolor=black, dotsize=0.4](0.19711548,-3.4028845)
			\psdots[linecolor=black, dotsize=0.4](2.5971155,-2.2028844)
			\psdots[linecolor=black, dotsize=0.4](3.3971155,-3.4028845)
			\psdots[linecolor=black, dotsize=0.4](1.7971154,-5.0028844)
			\psdots[linecolor=black, dotsize=0.4](3.3971155,-5.0028844)
			\psdots[linecolor=black, dotsize=0.4](2.5971155,-6.2028847)
			\psdots[linecolor=black, dotsize=0.4](4.1971154,-0.60288453)
			\psdots[linecolor=black, dotsize=0.4](4.9971156,-1.8028846)
			\psdots[linecolor=black, dotsize=0.4](4.9971156,-3.4028845)
			\psdots[linecolor=black, dotsize=0.4](4.1971154,-4.6028843)
			\psdots[linecolor=black, dotsize=0.4](3.3971155,-1.8028846)
			\psdots[linecolor=black, dotsize=0.4](3.3971155,-3.4028845)
			\psdots[linecolor=black, dotsize=0.4](5.7971153,-2.2028844)
			\psdots[linecolor=black, dotsize=0.4](6.5971155,-3.4028845)
			\psdots[linecolor=black, dotsize=0.4](4.9971156,-5.0028844)
			\psdots[linecolor=black, dotsize=0.4](6.5971155,-5.0028844)
			\psdots[linecolor=black, dotsize=0.4](5.7971153,-6.2028847)
			\psdots[linecolor=black, dotsize=0.1](7.3971157,-3.4028845)
			\psdots[linecolor=black, dotsize=0.1](7.7971153,-3.4028845)
			\psdots[linecolor=black, dotsize=0.1](8.197116,-3.4028845)
			\psdots[linecolor=black, dotsize=0.4](9.797115,-0.60288453)
			\psdots[linecolor=black, dotsize=0.4](10.5971155,-1.8028846)
			\psdots[linecolor=black, dotsize=0.4](10.5971155,-3.4028845)
			\psdots[linecolor=black, dotsize=0.4](9.797115,-4.6028843)
			\psdots[linecolor=black, dotsize=0.4](8.997115,-1.8028846)
			\psdots[linecolor=black, dotsize=0.4](8.997115,-3.4028845)
			\psdots[linecolor=black, dotsize=0.4](11.397116,-2.2028844)
			\psdots[linecolor=black, dotsize=0.4](12.197116,-3.4028845)
			\psdots[linecolor=black, dotsize=0.4](10.5971155,-5.0028844)
			\psdots[linecolor=black, dotsize=0.4](12.197116,-5.0028844)
			\psdots[linecolor=black, dotsize=0.4](11.397116,-6.2028847)
			\psline[linecolor=black, linewidth=0.08](0.9971155,-0.60288453)(1.7971154,-1.8028846)(1.7971154,-5.0028844)(2.5971155,-6.2028847)(3.3971155,-5.0028844)(3.3971155,-1.8028846)(4.1971154,-0.60288453)(4.9971156,-1.8028846)(4.9971156,-5.0028844)(5.7971153,-6.2028847)(6.5971155,-5.0028844)(6.5971155,-3.4028845)(5.7971153,-2.2028844)(4.9971156,-3.4028845)(4.1971154,-4.6028843)(2.5971155,-2.2028844)(0.9971155,-4.6028843)(0.19711548,-3.4028845)(0.19711548,-1.8028846)(0.9971155,-0.60288453)(0.9971155,-0.60288453)
			\psline[linecolor=black, linewidth=0.08](11.397116,-2.2028844)(9.797115,-4.6028843)(8.997115,-3.4028845)(8.997115,-1.8028846)(9.797115,-0.60288453)(10.5971155,-1.8028846)(10.5971155,-5.0028844)(11.397116,-6.2028847)(12.197116,-5.0028844)(12.197116,-3.4028845)(11.397116,-2.2028844)(11.397116,-2.2028844)
			\end{pspicture}
		}
	\end{minipage}
		\caption{Para-chain  $L_n$ and Meta-chain  $M_n$} \label{metaChainMn}
	\end{figure}
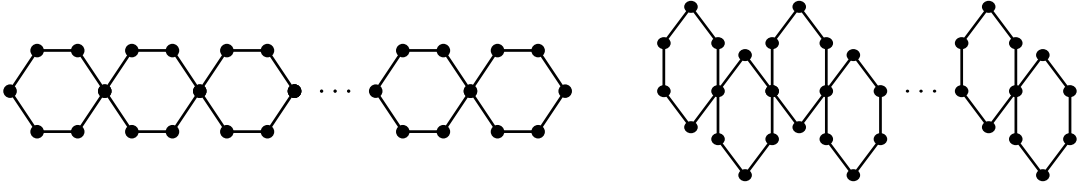

A cactus graph is a connected graph in which no edge lies in more than one cycle. Consequently,
each block of a cactus graph is either an edge or a cycle. If all blocks of a cactus $G$  are cycles of the same size $i$, the cactus is $i$-uniform.
A triangular cactus is a graph whose blocks are triangles, i.e., a $3$-uniform cactus.
A vertex shared by two or more triangles is called a cut-vertex. If each triangle of a triangular cactus $G$ has at most two cut-vertices, and each cut-vertex is shared by exactly two triangles,
we say that $G$ is a chain triangular cactus. By replacing triangles in this  definitions with  cycles of length $4$ we obtain cacti whose
every block is $C_4$.
We call such cacti square cacti. Note that the internal squares may differ in the way they connect to their neighbors. If their cut-vertices are adjacent, we say that such a square is an ortho-square; 
if the cut-vertices are not adjacent, we call the square a para-square \cite{cactus,Gutindex}.
	
	\begin{theorem}
		\begin{enumerate} 
			\item[(i)] 
			Let $T_n$ be the chain triangular graph (See Figure \ref{Chaintri}) of order $n$. Then for every $n\geq 2$,
$SO(T_n)=(4n-4)\sqrt{2}+4n\sqrt{5}.$
\item[(ii)] 
Let $Q_n$ be the para-chain square cactus graph (See Figure \ref{paraChainsqu}) of order $n$. Then for every $n\geq 2$,
$SO(Q_n)=8\sqrt{2}+(8n-8)\sqrt{5}.$
\item[(iii)] 
Let $O_n$ be the para-chain square cactus (See Figure \ref{ortho-ohn}) graph of order $n$. Then for every $n\geq 2$,
$SO(O_n)=(6n-4)\sqrt{2}+4n\sqrt{5}.$
\item[(iv)] 
 Let $O_n^h$ be the Ortho-chain graph (See Figure \ref{ortho-ohn}) of order $n$. Then for every $n\geq 2$,
$SO(O_n^h)=(10n-4)\sqrt{2}+4n\sqrt{5}.$
\item[(v)] 
Let $L_n$ be the para-chain hexagonal
cactus graph (See Figure \ref{metaChainMn}) of order $n$. Then for every $n\geq 2$,
$SO(L_n)=(4n+8)\sqrt{2}+(8n-8)\sqrt{5}.$
\item[(vi)] 
Let $M_n$ be the Meta-chain hexagonal
cactus graph (See Figure \ref{metaChainMn}) of order $n$. Then for every $n\geq 2$,
$SO(M_n)=(4n+8)\sqrt{2}+(8n-8)\sqrt{5}.$
\end{enumerate} 
	\end{theorem}
		\begin{proof}
			\begin{enumerate} 
				\item[(i)] 
	There are two edges with endpoints of degree $2$. Also there are $2n$ edges with endpoints of degree $2$ and $4$ and there are $n-2$ edges with endpoints of degree $4$. Therefore 
	$$SO(T_n)=2\sqrt{4+4}+2n\sqrt{4+16}+(n-2)\sqrt{16+16},$$
	and we have the result.	
	\item[(ii)]
	There are four edges with endpoints of degree $2$. Also there are $4n-4$ edges with endpoints of degree $2$ and $4$. Therefore 
	$$SO(Q_n)=4\sqrt{4+4}+(4n-4)\sqrt{4+16},$$
	and the result follows.			
	\item[(iii)]	
	There are $n+2$ edges with endpoints of degree $2$. Also there are $2n$ edges with endpoints of degree $2$ and $4$ and there are $n-2$ edges with endpoints of degree $4$. Therefore 
	$$SO(O_n)=(n+2)\sqrt{4+4}+2n\sqrt{4+16}+(n-2)\sqrt{16+16},$$
	and we have the result.	
		\item[(iv)] 
	There are $3n+2$ edges with endpoints of degree $2$. Also there are $2n$ edges with endpoints of degree 2 and 4 and there are $n-2$ edges with endpoints of degree 4. Therefore 
	$$SO(O_n^h)=(3n+2)\sqrt{4+4}+2n\sqrt{4+16}+(n-2)\sqrt{16+16},$$
		and the result follows.
\item[(v)] 
	There are $2n+4$ edges with endpoints of degree $2$. Also there are $4n-4$ edges with endpoints of degree 2 and 4. Therefore 
	$$SO(L_n)=(2n+4)\sqrt{4+4}+(4n-4)\sqrt{4+16},$$
	and we have the result.	
	\item[(vi)] 
	There are $2n+4$ edges with endpoints of degree $2$. Also there are $4n-4$ edges with endpoints of degree $2$ and $4$. Therefore 
	$$SO(M_n)=(2n+4)\sqrt{4+4}+(4n-4)\sqrt{4+16},$$
		and the result follows. \qed
	\end{enumerate} 
		\end{proof}

\begin{corollary}
	Meta-chain hexagonal
	cactus graphs  and para-chain hexagonal
	cactus graphs of the same order, have the same Sombor index.
\end{corollary}

\section{Sombor index of some operations on a graph}
In this section, we  examine the effects on $SO(G)$ when $G$ is modified by operations on vertex and  edge of $G$.

	\begin{theorem}
		Let $G=(V,E)$ be a graph and $e=uv\in E$. Also let $d_w$ be the degree of vertex $w$ in $G$. Then,
		$$SO(G-e) < SO(G) - \frac{|d_u-d_v|}{\sqrt{2}}. $$
	\end{theorem}
	
	\begin{proof}
		First we remove edge $e$ and find $SO(G-e)$. Now Obviously, by adding edge $e$ to $G-e$ and $\sqrt{d_u^2+d_v^2}$ to $SO(G-e)$, then $SO(G)$ is greater than that. Since $\sqrt{a^2+b^2}\geq \frac{|a-b|}{\sqrt{2}}$, then
		$$SO(G) > SO(G-e) + \sqrt{d_u^2+d_v^2} \geq SO(G-e) + \frac{|d_u-d_v|}{\sqrt{2}}, $$
		and therefore we have the result.
		\qed
	\end{proof}
	
	\begin{theorem}
		Let $G=(V,E)$ be a graph and $v\in V$. Also let $d_u$ be the degree of vertex $u$ in $G$. Then,
		$$SO(G-v) < SO(G) - \sum_{uv\in E}\frac{|d_u-d_v|}{\sqrt{2}}. $$
	\end{theorem}
	
	\begin{proof}
		First we remove vertex $v$ and all edges related that. Then, find $SO(G-v)$. Now Obviously, by adding vertex $v$ and all edges related that to $G-v$ and $\sum_{uv\in E}\sqrt{d_u^2+d_v^2}$ to $SO(G-v)$, then $SO(G)$ is greater than that. Since $\sqrt{a^2+b^2}\geq \frac{|a-b|}{\sqrt{2}}$, then
		$$SO(G) > SO(G-v) + \sum_{uv\in E}\sqrt{d_u^2+d_v^2} \geq SO(G-v) + \sum_{uv\in E}\frac{|d_u-d_v|}{\sqrt{2}}, $$
		and therefore we have the result.
		\qed
	\end{proof}

	For any $k \in \mathbb{N}$, the $k$-subdivision of $G$ is a simple graph $G^{\frac{1}{k}}$ which is constructed by replacing each edge of $G$ with a path of length $k$. The following theorem is about the Sombor index of $k$-subdivision of graph $G$. 
	
	\begin{theorem}\label{thm-fracG}
		Let $G=(V,E)$ be a graph with $|V|=n$ and $|E|=m$. For every $k\geq 2$,
		\begin{align*}
		SO(G^\frac{1}{k}) = 2m(k-2)\sqrt{2}+ \displaystyle\sum_{u\in V}d_u\sqrt{d_u^2+4}.
		\end{align*}
	\end{theorem}
	\begin{proof}
		There are $d_u$ edges incident $u\in V$ with endpoints of degree $d_u$ and $2$ in $G^\frac{1}{k}$. Also there are $m(k-2)$ edges with endpoints of degree 2 in that. So
		\begin{align*}
		SO(G^\frac{1}{k}) = m(k-2)\sqrt{4+4}+ \displaystyle\sum_{u\in V}d_u\sqrt{d_u^2+4},
		\end{align*}
		and we have the result.
		\qed
	\end{proof}

	As a result of the Theorem \ref{thm-fracG}, we have:
	
	\begin{corollary}\label{cor-fracG}
		Let $G=(V,E)$ be a graph of order $n$ and size $m$ and. If  $\Delta$ is the maximum degree   and $\delta$ is the minimum degree of vertices in $G$, then for every $k\geq 2$,
		\begin{align*}
		2m(k-2)\sqrt{2}+ n\delta\sqrt{\delta^2+4}	\leq SO(G^\frac{1}{k}) \leq 2m(k-2)\sqrt{2}+ n\Delta\sqrt{\Delta^2+4}.
		\end{align*}
	\end{corollary}

	\begin{remark}
		The bounds in the Corollary \ref{cor-fracG} are sharp. It suffices to consider cycle graph or complete graph. 
	\end{remark}

	For a given a graph parameter $f(G)$ and a positive integer $n$, the well-known Nordhaus-Gaddum problem is to determine sharp bounds for $f(G)+f(\overline{G})$ and $f(G)f(\overline{G})$ over the class of connected graph $G$, with order $n$, size $m$.   Many Nordhaus-–Gaddum type relations have attracted considerable attention in graph theory. Comprehensive results regarding this topic can be found in e.g., \cite{Mao,Shang}.
	
	\begin{theorem}
		Let $G=(V,E)$ be a graph with $|V|=n$. Also let $d_u$ be the degree of vertex $u$ in $G$. Then,
		\begin{align*}
		SO(G)+SO(\overline{G}) \geq  \displaystyle\sum_{u,v\in V}\frac{|d_u-d_v|}{\sqrt{2}}.
		\end{align*}
	\end{theorem}
	
	\begin{proof}
		By the definition of Sombor index for the graph $G$ we have:
		$$SO(G) =\sum_{uv\in E}\sqrt{d_u^2+d_v^2},$$
		and
		$$SO(\overline{G}) =\sum_{uv\notin E}\sqrt{(n-1-d_u)^2+(n-1-d_v)^2}.$$	
		Since $\sqrt{a^2+b^2}\geq \frac{|a-b|}{\sqrt{2}}$, then
		\begin{align*}
		SO(G)+SO(\overline{G}) &\geq  \displaystyle\sum_{uv\in E}\frac{|d_u-d_v|}{\sqrt{2}}+ \displaystyle\sum_{uv\notin E}\frac{|d_u-d_v|}{\sqrt{2}}\\
		&=\sum_{u,v\in V}\frac{|d_u-d_v|}{\sqrt{2}}.
		\end{align*}
		Therefore we have the result.
		\qed
	\end{proof}

\section{Sombor index of join and corona of two graphs} 
	In this section, we study the  Sombor index of join product and corona product of two graphs.

	\begin{theorem}
Let $G=(V_G,E_G)$ and $H=(V_H,E_H)$ be two graphs with $|V_G|=n$ and $|V_H|=m$. Also let $d_u$ be the degree of vertex $u$ in $G$ and $H$ before joining of two graphs. Then,
	\begin{align*}
	SO(G+H) &\geq \displaystyle\sum_{u\in G,v\in H}\frac{|d_u-d_v+m-n|}{\sqrt{2}}\\
	&\quad+ \displaystyle\sum_{uv\in E_H}\frac{|d_u-d_v|}{\sqrt{2}}+ \displaystyle\sum_{uv\in E_G}\frac{|d_u-d_v|}{\sqrt{2}}.
	\end{align*}
	\end{theorem}

	\begin{proof}
	By the definition of join of two graphs and the definition of Sombor index, we have 
		\begin{align*}
	SO(G+H) &= \displaystyle\sum_{u\in G,v\in H}\sqrt{(d_u+m)^2+(d_v+n)^2}\\
	&\quad+ \displaystyle\sum_{uv\in E_H}\sqrt{(d_u+m)^2+(d_v+m)^2} \\
	&\quad+ \displaystyle\sum_{uv\in E_G}\sqrt{(d_u+n)^2+(d_v+n)^2},
	\end{align*}
	Since $\sqrt{a^2+b^2}\geq \frac{|a-b|}{\sqrt{2}}$, then we have the result.
	\qed
	\end{proof}

The following Theorem gives the values of Sombor index for corona of $P_n$ and $C_n$ with $K_1$:

\begin{theorem}
	\begin{enumerate} 
		\item[(i)] $SO(P_n\circ K_1)=(3n-9)\sqrt{2}+(n-2)\sqrt{10}+2\sqrt{5}+2\sqrt{13}.$
		\item[(ii)] $SO(C_n\circ K_1)=3n\sqrt{2}+n\sqrt{10}.$
	\end{enumerate} 
\end{theorem}
\begin{proof}
	\begin{enumerate} 
		\item[(i)]	
		There are two edges with endpoints of degree $1$ and $2$ and two edges with end points of degree $2$ and $3$ . Also there are $n-2$ edges with endpoints of degree $1$ and $3$ and there are $n-3$ edges with endpoints of degree $3$. Therefore 
		$$SO(P_n\circ K_1)=(n-3)\sqrt{9+9}+(n-2)\sqrt{1+9}+2\sqrt{1+4}+2\sqrt{4+9},$$
		and the result follows. 
		\item[(ii)] 
		There are $n$ edges with endpoints of degree $1$ and $3$ and there are $n$ edges with endpoints of degree $3$.  Therefore 
		$$SO(C_n\circ K_1)=n\sqrt{1+9}+n\sqrt{9+9},$$
		and we have the result.	\qed
	\end{enumerate} 
\end{proof}

In the following result, we present a lower bound for the corona of two graphs $G$ and $H$.

	\begin{theorem}
Let $G=(V_G,E_G)$ and $H=(V_H,E_H)$ be two graphs with $|V_H|=m$. Also let $d_u$ be the degree of vertex $u$ in $G$ and $H$ before corona of two graphs. Then,
	\begin{align*}
	SO(G\circ H) &\geq \displaystyle\sum_{u\in G,v\in H}\frac{m|d_u-d_v+m-1|}{\sqrt{2}}\\
	&\quad+ \displaystyle\sum_{uv\in E_H}\frac{|d_u-d_v|}{\sqrt{2}}+ \displaystyle\sum_{uv\in E_G}\frac{|d_u-d_v|}{\sqrt{2}}.
	\end{align*}
	\end{theorem}

	\begin{proof}
	By the definition of corona of two graphs and the definition of Sombor index, we have 
		\begin{align*}
	SO(G\circ H) &= \displaystyle\sum_{u\in G,v\in H}m\sqrt{(d_u+m)^2+(d_v+1)^2}\\
	&\quad+ \displaystyle\sum_{uv\in E_H}\sqrt{(d_u+1)^2+(d_v+1)^2} \\
	&\quad+ \displaystyle\sum_{uv\in E_G}\sqrt{(d_u+m)^2+(d_v+m)^2},
	\end{align*}
	Since $\sqrt{a^2+b^2}\geq \frac{|a-b|}{\sqrt{2}}$, then we have the result.
	\qed
	\end{proof}

\section{Acknowledgements} 

The first author would like to thank the Research Council of Norway (NFR Toppforsk Project Number 274526, Parameterized Complexity for Practical Computing) and Department of Informatics, University of
Bergen for their support. Also he is thankful to Michael Fellows and
Michal Walicki for conversations.

\end{document}